\documentclass[12pt,reqno]{article}

\usepackage[usenames]{color}
\usepackage{amssymb}
\usepackage{graphicx}
\usepackage{amscd}

\usepackage[colorlinks=true,
linkcolor=webgreen,
filecolor=webbrown,
citecolor=webgreen]{hyperref}

\definecolor{webgreen}{rgb}{0,.5,0}
\definecolor{webbrown}{rgb}{.6,0,0}

\usepackage{color}
\usepackage{fullpage}
\usepackage{float}

\usepackage{graphics,amsmath,amssymb}
\usepackage{amsthm}
\usepackage{amsfonts}
\usepackage{latexsym}
\usepackage{epsf}

\usepackage{verbatim}

\usepackage{tikz}
\usepackage{forest}

\newcommand{\seqnum}[1]{\href{http://oeis.org/#1}{\underline{#1}}}

\begin{document}

%\begin{center}
%\epsfxsize=4in
%\leavevmode\epsffile{logo129.eps}
%\end{center}

\theoremstyle{plain}
\newtheorem{theorem}{Theorem}
\newtheorem{corollary}[theorem]{Corollary}
\newtheorem{lemma}[theorem]{Lemma}
\newtheorem{proposition}[theorem]{Proposition}

\theoremstyle{definition}
\newtheorem{definition}[theorem]{Definition}
\newtheorem{example}[theorem]{Example}
\newtheorem{conjecture}[theorem]{Conjecture}
\newtheorem{identity}{Identity}

\theoremstyle{remark}
\newtheorem{remark}[theorem]{Remark}

\begin{center}
\vskip 1cm{\LARGE\bf 
Variants of Base 3 over 2
}
\vskip 1cm
\large
Matvey Borodin, Hannah Han, Kaylee Ji, Alexander Peng, David Sun, Isabel Tu, Jason Yang, William Yang, Kevin Zhang, Kevin Zhao \\ 
PRIMES STEP, Dept. Math\\
MIT\\
77 Mass. Ave \\
Cambridge,  MA 02139  \\
USA \\
\href{mailto: primes.step@gmail.com}{\tt primes.step@gmail.com}
\ \\
Tanya Khovanova \\
Department of Mathematics\\ 
MIT\\
77 Mass. Ave \\
Cambridge,  MA 02139  \\
USA \\
\href{mailto: tanyakh@yahoo.com}{\tt tanyakh@yahoo.com}
\end{center}

%\vskip .2 in

%\maketitle

\begin{abstract}
We discuss two different systems of number representations that both can be called \textit{base 3/2}. We explain how they are connected. Unlike classical fractional extension, these two systems provide a finite representation for integers. We also discuss a connection between these systems and 3-free sequences.
\end{abstract}

%%%%%%%%%%%%%%%%%%%%%%%%%%%%%%%%%%%%%%%%%%%%%%%%%%%%%%%%%%%%%%%%%%%%%%%%%%%%%%%%

\section{Introduction}

A traditional non-integer base $\beta$ was explored by R\'{e}nyi \cite{R} and Parry \cite{P}. It represented numbers using digits not-exceeding $\beta$. Every nonnegative real number can be represented as a string of digits, usually using the radix point in such bases. Integers are usually represented as infinite strings.

A different cool concept called \textit{exploding dots} was invented by Propp \cite{JP} and popularized by Tanton \cite{JT}. For a rational base $b/a$ it allows using digits below $a+b$. The advantage of this approach is that integers can be represented by finite strings. These bases were thoroughly studied by Akiyama and others in \cite{AFS}.

In this paper, we are interested in base 3/2, which represents integers using digits 0, 1, and 2. We discovered sequence A256785 in the OEIS \cite{OEIS}, which uses digits 0 and 1, and symbol H to represent integers. We call this base, base 1.5 to differentiate it from base 3/2. We discovered an isomorphism between the two bases.

While writing the results, we stumbled on another sequence, A265316, that was even more surprising. Consider the following sequence: Take even numbers written in base 3/2 using exploding dots with digits 0, 1, and 2. Then interpret the result in ternary. When we plugged in the results we got sequence A265316: Consider a greedy way to divide non-negative integers into an infinite set of sequences not containing a 3-term arithmetic progression. The sequence A265316 is the set of first elements of these sequences.

Here is how this paper is arranged. In Section~\ref{sec:explodingdots} we introduce exploding dots. In Section~\ref{sec:base32} we describe a particular case of exploding dots called $2 \leftarrow 3$ machine, corresponding to base 3/2. This base uses digits 0, 1, and 2 in their expansions. In Section~\ref{sec:base15} we discuss the base 1.5 introduced in sequence A256785 which uses digits 0 and 1, and symbol H.

In Section~\ref{sec:mysteryseq} we define sequence A265316 and discuss its connections to the base 3/2. We do not completely prove the fact that these sequences are the same, but we prove a lot of properties for both sequences. 

In Section~\ref{sec:differentways} we explore several natural ways to represent the same number in base 1.5. In Section~\ref{sec:isomorphism} we produce an isomorphism between the two bases 1.5 and 3/2.

This research was done by the PRIMES STEP junior group. PRIMES STEP is a program based at MIT for students in grades 6-9 to try research in mathematics. 

\section{Exploding Dots}\label{sec:explodingdots}

Here we explain exploding dots. We start with a row of boxes that can be extended to the left. We label the boxes from right to left. The rightmost box is labeled zero. The second one to the right is box 1, the third to the right is box 2 and so on.

We also have an integer $b$ that is our base. Consider integer $N$. To find its value in base $b$, we place $N$ dots in box 0. Now we allow explosions. As soon as there are $b$ dots in box $k$, they, BOOM, explode. That means we remove $b$ dots from box $k$ and add one dot in the box to the left, which is of course, numbered $k+1$. We continue exploding until nothing can explode anymore, meaning each box has no more than $b$ dots. This process is called a $1 \leftarrow b$ machine. At the end, we write the number of dots in each box from left to right, dropping the leading zeros. The result is the representation of integer $N$ in base $b$.

For example, to calculate 5 in base 2, we start with 5 dots in the rightmost box, box 0. We can represent this state of our machine as 5. Since we have more than two dots, each pair of dots explodes adding a dot to the box directly to the left. As there are two pairs, we add two dots to box 1 and remove 4 dots from box 0. We can represent the result as 21: one dot in the rightmost box and two dots in the box to the left. Now there are two dots together in box 1; therefore, we have another BOOM, which results in base-2 representation of 5: $5_{10} = 101_2$, see Figure~\ref{fig:base2}.

\begin{figure}[ht]
\centering
\includegraphics[scale=0.7]{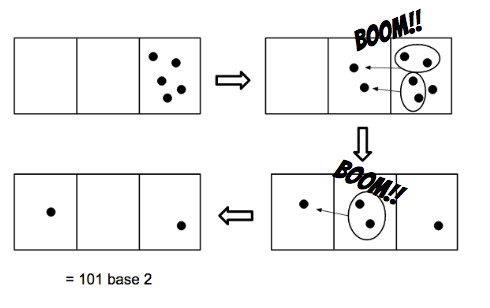}
\caption{Exploding dots show how to represent 5 in base 2}\label{fig:base2}
\end{figure}

The cool part about the exploding-dots machines is that they are easily generalizable to rational bases. The $a \leftarrow b$ machine is a machine where each time there are at least $b$ dots in a box, there is an explosion. An explosion in box $k$ wipes out $b$ dots from box $k$, while adding $a$ dots to box $k+1$. To represent an integer $N$, we start with $N$ dots in box zero. After the fireworks are over, that is all boxes have fewer than $b$ dots, we read the number of dots from left to right starting with the left most non-empty box. The result is the representation of $N$ in base $b/a$. We number the digits of this representation similar to the way we number boxes: from right to left: $d_k d_{k-1}\ldots d_1d_0$.

For example, to calculate 5 in base 3/2, we start with 5 dots in the rightmost box, box 0. We can represent this state of our machine as 5. Since we have more than three dots we have an explosion: the number of dots in the rightmost box decreases by 3 and we add 2 dots to the box on the left. The result is 22: which is the base-3/2 representation of 5: $5_{10} = 22_{3/2}$, see Figure~\ref{fig:base3over2}.

\begin{figure}[ht]
\centering
\includegraphics[scale=0.6]{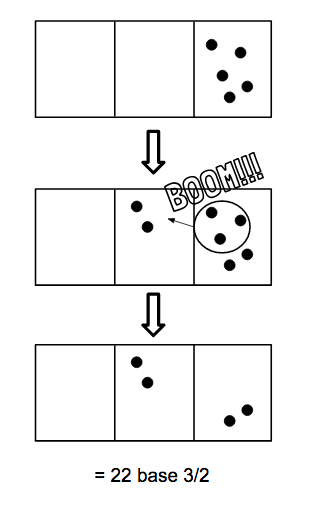}
\caption{Exploding dots show how to represent 5 in base 3/2.}\label{fig:base3over2}
\end{figure}

\section{Base 3/2}\label{sec:base32}

The $2 \leftarrow 3$ machine is a machine where three dots explode generating two new dots in the box on the left. 

For example, number 7 in base 3/2 becomes 211. 

The first several numbers written in base 3/2 form sequence A024629 in the OEIS \cite{OEIS}: \[0, 1, 2, 20, 21, 22, 210, 211, 212, 2100, \ldots.\]

Here are some awesome properties of base 3/2 \cite{JP,JT}:

\begin{itemize}
\item Every integer only uses digits 0, 1, and 2.
\item Every number starting with 2 starts with 2.
\item Every number starting with 8 starts with 21 followed by either 0 or 2.
\item The last digit repeats in a cycle of 3, the last two digits repeat in a cycle of 9, and so on: the last $k$ digits repeat in a cycle of $3^k$. 
\item Removing one or several last digits of an integer in this base gives another integer in the base.
\end{itemize}

It is interesting to note that base 6/4 is different from base 3/2. For example, numbers in base 6/4 can have 5 as a digit, while numbers in base 3/2 can not. For this reason, it is important not to reduce the fraction to simplest terms in this definition of the base. In particular, it is important to call this base, base 3/2, not base 1.5.

The digits in base $b/a$ represent how the integer $N$ can be decomposed into powers of $3/2$ \cite{JP}.

\begin{lemma}
If $d_kd_{k-1}\ldots d_1d_0$ is a representation of integer $N$ in base $3/2$, then 
\[N = \sum_{i=0}^k d_i\frac{3^i}{2^i}.\]
\end{lemma}

\section{Base 1.5}\label{sec:base15}

A definition of base 1.5 is given in sequence A256785 in the OEIS \cite{OEIS}. This base uses three symbols: 0, 1, and H. Symbol H represents 0.5. Letter H was likely chosen because of the word \textbf{h}alf. For emphasis, we use 1.5 instead of $\frac{3}{2}$ in this section where appropriate.

Here are a few rational numbers using these three digits in ascending order of the number values: 
\[\text{0, H, H0, 1, H00, HH, 10, H0H, H000, H1, HH0, 1H, H01, H00H, 100, }
%HHH, H0H0, 11, H0000, H10, H001, 10H, 
\ldots.\]

We call this sequence the \textit{ascending} sequence and denote it as $A_n$. The corresponding values are:
\[0, 0.5, 0.75, 1, 1.125, 1.25, 1.5, 1.625, 1.6875, 1.75, 1.875, 2, 2.125, 2.1875, 2.25,
\ldots.\]

One might wonder how it could be possible to write this sequence: that is, why are we always able to find the next number in value in an infinite set of numbers? The smallest number with $j$ digits is H00...0: it has $j-1$ zeros and the value of $0.5\cdot 1.5^{j-1}$. Since this value increases as $j$ increases, to find all numbers that are less than $0.5\cdot 1.5^{j-1}$, we only need to have a finite check of all the numbers with less than $j$ digits.

An oddity of this base, although expected, is that not all of these numbers are integers. The indices of integers in this sequence are:
\[0, 3, 11, 25, 46, 77, 117, 169, 232, 308, 401, 508, 631, 771, 929, 1108, 1308, \ldots,\]
which is now sequence A320035.

%0, 3, 11, 25, 46, 77, 117, 169, 232, 308, 401, 508, 631, 771, 929, 1108, 1308, 1527, 1767, 2029, 2315, 2626, 2961, 3325, 3719, 4138, 4585, 5057, 5561, 6094, 6658, 7251, 7880, 8543, 9245, 9982, 10760, 11572, 12419, 13305, 14226, 15181, 16177, 17209, 18285, 19404, 20560, 21760, 23007, 24297, 25637, 27027, ...

The first few natural numbers written in this base are:

\[1 =1, \ 2 = 1\text{H}, \ 3 = 1\text{H}0, \ 4 = 1\text{H}1, \ 5 = 1\text{H}0\text{H}, \ 6 = 1\text{H}10, 7 = 1\text{H}11.\]

The weird thing about base 1.5 is that an $i$-digit number might be smaller than a $j$-digit number where $i > j$. 

Other than the ascending order, there is another reasonable order to write these numbers in: we call it the \textit{dictionary} order. Consider numbers that use only zeros and two other digits $a < b$. Write the numbers in the increasing order. Replace $a$ by H, and $b$ by 1. In this order, the numbers with more digits will go after the numbers with fewer digits. Did we mention that this base is weird? The sequence of base 1.5 rational numbers in the dictionary order is sequence $B_n$:

\[\text{0, H, 1, H0, HH, H1, 10, 1H, 11, H00, H0H, H01, HH0, HHH, HH1, }\ldots.\]

The corresponding values are:

\[0, 0.5, 1, 0.75, 1.25, 1.75, 1.5, 2, 2.5, 1.125, 1.625, 2.125, 1.875, 2.375, 2.875, \ldots.\]

The indices of integers in the dictionary ordered sequence are: 
\[0, 2, 7, 21, 23, 64, 69, 71, 193, 207, \ldots.\]

This is the sequence A265316. The sequence A265316 is not related to any base. We will discuss this unexpected connection in Section~\ref{sec:mysteryseq}.

The values in the dictionary order sequence go up and down, whereas if it were an integer base, they would always go up. To be precise, the dictionary order sequence goes up, up, down in a cyclic manner. 

\begin{lemma}
For integer $k \geq 0$, $B_{3k}<B_{3k+1}<B_{3k+2}$, while $B_{3k+2} > B_{3k+3}$.
\end{lemma}

\begin{proof}
Numbers $B_{3k}$, $B_{3k+1}$, and $B_{3k+2}$ only differ in the last digit. Therefore, we have $B_{3k+2} = B_{3k+1}+ 1/2 = B_{3k}+1$. This proves the first part of the lemma. 

For the second part we look at the last two digits. If the last two digits of $B_{3k+2}$ are 01, then $B_{3k+3}$ has the same prefix and the last two digits H0. The statement follows from the fact that $\text{H}0 < 1$. If the last two digits of $B_{3k+2}$ are H1, then $B_{3k+3}$ has the same prefix and the last two digits 10. The statement follows from the fact that $10 < \text{H}1$.

Suppose the last two digits of $B_{3k+2}$ are 11. Let us assume that $B_{3k+2}$ ends with $k \geq 2$ ones. We denote the digit before the last run of ones in $B_{3k+2}$ as $b$, where $b$ is either 0 or H. Then number $B_{3k+3}$ differs from $B_{3k+2}$ only in the last $k+1$ digits. The last $k+1$ digits of $B_{3k+2}$ are $b+0.5$ followed by $k$ zeros. Therefore the difference $B_{3k+2}-B_{3k+3}$ is:
\[1.5^{k-1} + 1.5^{k-2} + \cdots + 1.5^{1} + 1.5^{0} -  \frac{1}{2} \cdot 1.5^{k}.\]
Summing the geometric series we get
\[2 \left(1.5^{k} - 1\right) - \frac{1}{2} \cdot 1.5^{k} = 1.5^{k+1} -2.\]
The fact that $k \geq 2$ means the difference is positive.
\end{proof}

We want to introduce some marvelous sequences that show the connection between the ascending order and the dictionary order. The first sequence shows the value order when the numbers are arranged in the dictionary order. In other words, our sequence $a(n)$ is such that $a(n) = k$, if $A_k=B_n$. This is always possible because the sequences $A_n$ and $B_n$ contain the same numbers, just in a different permutation. This sequence is now A320274: 
\[0, 1, 3, 2, 5, 9, 6, 11, 17, 4, 7, 12, 10, 15, 23, 19, 27, 37, 14, 21, 29, 25, 34, 46,
%40, 53, 68, 8, 13, 20, 16, 24, 33, 28, 39, 51, 22, 31, 43, 36, 48, 63, 56, 71, 88, 45, 
 \ldots.\]

Similarly, we can define sequence $b(n)$ so that $b(n) = k$, if $B_k=A_n$. This is now sequence A320273:
\[0, 1, 3, 2, 9, 4, 6, 10, 27, 5, 12, 7, 11, 28, 18, 13, 30, 8, 81, 15, 29, 19, 
%36, 14, 31, 21, 82, 16, 33, 20, 84, 37, 54, 32, 22, 83, 39, 17, 90, 34,
 \ldots.\]

The two sequences above are permutations of non-negative integers. Therefore, they contain every number. By definition, they are inverses of each other.

%hh00, hh1, h0hh, 1h0, h000h, h1h, h010, 101, h00h0, hh0h, 1000, h0h1, 1hh, h0001, hhh0, h11, h0h00, h01h, 110, h00hh, hh01, 100h, h100, 1h1, h0010, hhhh, 10h0, h0h0h, h011, hh000, 11h, h00h1, hh10, 1001, h0hh0, h10h, 1h00, h001h, hhh1, 10hh, h0h01, h1h0, hh00h, 111, h0100, hh1h, 1010, h0hhh, h101, hh0h0, 1h0h, h0011, 10000, 10h1, h0h10, h1hh, hh001, 1hh0, h010h, hh11, hhh00, 101h, h0hh1, h110, hh0hh, 1h01, h01h0, 1000h, 1100, h0h1h, h1h1, hh010, 1hhh, h0101, 100h0, hhh0h, 1011, h1000, h11h, hh0h1, 1h10, h01hh, 10001, hhhh0, 110h, h0h11, 10h00, hh01h, 1hh1, h0110, 100hh, hhh01, 11h0, h100h, h111, hh100, 1h1h, h01h1, 10010, hhhhh, 1101, h10h0, 10h0h, hh011, 1h000, h011h, 100h1, hhh10, 11hh, h1001, 10hh0, hh10h, 1h11, h1h00, 1001h, hhhh1, 1110, h10hh, 10h01, hh1h0, 1h00h, h0111, 10100, hhh1h, 11h1, h1010, 10hhh, hh101, 1h0h0, h1h0h, 10011, 111h, h10h1, 10h10, hh1hh, 1h001, h1hh0, 1010h, hhh11, 1hh00, h101h, 10hh1, hh110, 1h0hh, h1h01, 101h0, 1111, h1100, 10h1h, hh1h1, 1h010, h1hhh, 10101, 1hh0h, h1011, 11000, hh11h, 1h0h1, h1h10, 101hh, 1hhh0, h110h, 10h11, 1h01h, h1hh1, 10110, 1hh01, h11h0, 1100h, hh111, 1h100, h1h1h, 101h1, 1hhhh, h1101, 110h0, 1h011, 1011h, 1hh10, h11hh, 11001, 1h10h, h1h11, 11h00, 1hhh1, h1110, 110hh, 1h1h0, 10111, 1hh1h, h11h1, 11010, 1h101, 11h0h, h111h, 110h1, 1h1hh, 11hh0, 1hh11, 1101h, 1h110, 11h01, h1111, 11100, 1h1h1, 11hhh, 11011, 1h11h, 11h10, 1110h, 11hh1, 111h0, 1h111, 11h1h, 11101, 111hh, 11h11, 11110, 111h1, 1111h, 11111

\section{The mysteries of sequence A265316}\label{sec:mysteryseq}

\subsection{The definition of A265316}

Now we go back to sequence A265316, which appeared here as indices of integers when numbers containing digits 0, 1, and 2 are written in the dictionary order and interpreted in base 3/2. We call this sequence the \textit{Stanley cross-sequence}: 
\[0, 2, 7, 21, 23, 64, 69, 71, 193, 207, 209, 214, \ldots.\]

Before providing the official definition of the sequence, we give several other definitions. A \textit{3-free} sequence is an integer sequence with no three elements forming an arithmetic progression. Given a start of a sequence of non-negative integers, the \textit{Stanley sequence} is a lexicographically smallest 3-free sequence with the given start \cite{OS}. The simplest Stanley sequence is the one that starts with 0, 1. It is sequence A005836 in the OEIS \cite{OEIS}: 
\[0, 1, 3, 4, 9, 10, 12, 13, 27, 28, 30, \ldots.\]

Now we are ready to give a description of sequence A265316 from the OEIS \cite{OEIS}. 

\begin{enumerate}
\item Consider the simplest Stanley sequence: 0, 1, 3, 4, 9, 10 and so on. We denote this sequence $S_0$. This sequence can be described as non-negative integers that do not contain 2 in their ternary representation.  %More terms:
%0, 1, 3, 4, 9, 10, 12, 13, 27, 28, 30, 31, 36, 37, 39, 40, 81, 82, 84, 85,
\item Then we use the leftover integers and build a new minimal 3-free sequence. The new sequence is 2, 5, 6, 11, 14 and so on. This sequence is now sequence A323398 in the OEIS. We denote this new sequence $S_1$.
\item Then we exclude this sequence too and continue building a new greedy 3-free sequence $S_2$: 7, 8, 16, 17, 19, 20, 34, and so on. This sequence is now sequence A323418 in the OEIS.
\item We continue this procedure to the new sequence $S_3$: 21, 22, 24, 25, 48, 49, 51, and so on, which is now sequence A323419 in the OEIS.
\item It is known \cite{R} that 3-free sequences have density zero. Therefore, we can build an infinite number of such sequences. The starting numbers of these series of sequences form sequence A265316 which is the object of this section. That is A265316$(n)$ is the first term of $S_n$.
\end{enumerate}

\subsection{Greedy 3-free sequences in base 3/2}

We now want to repeat the procedure of building 3-free sequences in base 3/2 using not just integers, but all finite strings containing three digits 0, 1, and 2. We call these numbers \textit{integer-like} numbers.  

It is wildly known \cite{OS} that the lexicographically first 3-free sequence, that is the simplest Stanley sequence, are numbers that are represented in base 3 without twos.

Our situation is similar and different at the same time. Integer-like numbers that are represented in base 3/2 have different values than the same strings interpreted in base 3. Also, there are two different natural orders on all integer-like numbers written with 0, 1, and 2. One is the value order if they are interpreted in base 3 or ten, and the other one when they are interpreted in base 3/2. The second order is different from the first. For example, $10 > 2$ in the first order and $10 < 2$ in the second. The first order is the dictionary order we described before. The good news: the numbers without twos will be ordered the same way in both orders.

We want to show that the lexicographically first sequence in integer-like numbers independently of which order, base 3 value or base 3/2 value, we choose is the same sequence: 

\begin{lemma}
The sequence of integer-like numbers in base 3/2 that does not contain twos in their base 3/2 interpretation is a 3-free sequence. Moreover, this sequence is the lexicographically first 3-free sequence in both orders.
\end{lemma}

\begin{proof} The first part of the proof is similar to the corresponding proof for the Stanley sequence starting 0, 1. 

Any integer-like number $x$ that has a digit 2 in base 3/2 can be written in the form $2b-a$, where $a$ and $b$ are integer-like numbers without a two in their $3/2$ representation and $b>a$. For example, $x=20211022021220220121111021 = 2\cdot 10111011011110110111111011-00011000001000000101111001$.
We can choose $b$ by replacing ones in $x$ with zeros and changing twos to ones; we can choose $a$ by changing all the digits 2 in $x$ to zero. Notice that $a,b < x$ in both orders.

Next, no three different integer-like numbers without a 2 in their base 3/2 representation can be in an arithmetic progression. To see why, suppose they were in such a progression. Let the numbers be $a$, $b$, and $2b-a$. If $a$ and $b$ are without a 2, then $2b$ is all 2s and 0s, and $2b$ would need all its digit 2s lined up with all of $a$'s digit 1s for $2b-a$ to not have a 2 remaining after subtraction. But then $a$ and $b$ are the same number, leading to a contradiction. This is the same argument as in base 10. As there are no carries in the argument, the argument works in any base.

We showed that the sequence of integer-like numbers without twos in base 3/2 is a 3-free sequence. Now we need to show that it is lexicographically the first in both orderings. 

The sequence starts with 0 and 1 in both cases.

We continue by induction. Assume by induction that the first $n$ numbers are the integer-like numbers without a 2 in base 3/2, the $n$-th number being $y$. We know that the next number $z$ without a 2 is valid, and we must prove it is of the smallest value in both orderings. Suppose it is not of minimum value, and the next term is instead a number $x$ between $y$ and $z$. Then $x$ must contain a 2. As we saw before, $x$ can be represented as $2b-a$ for $b$ and $a$ two numbers without a 2. As $a, b<x$ in both orderings, then $a$ and $b$ must both be among the first $n$ terms of the sequence. Then $a$, $b$, $x$ form a 3-term arithmetic progression, leading to a contradiction. Therefore, the next 2-free integer-like number when written in base 3/2 is the next term of the sequence.
\end{proof}

We denote the sequence of integer-like numbers that contain only zeros and ones in base 3/2 and arranged in the dictionary order as $\mathcal{S}_0$. Now we want to consider a set of sequences $\mathcal{S}_k$, where $\mathcal{S}_k = \mathcal{S}_0 +2k$ in base 3/2. We show several properties of these sequences.

\begin{lemma}
\begin{enumerate}
\item Each sequence $\mathcal{S}_k$ is 3-free.
\item Sequences $\mathcal{S}_k$ do not overlap.
\item Every integer-like number belongs to one of the sequences.
\item $\mathcal{S}_k$ is the lexicographically first sequence with no 3-term arithmetic progression chosen from the set of numbers $\cup_{i \geq k} \mathcal{S}_i$ when we use the value order.
\end{enumerate}
\end{lemma}

\begin{proof}
1. Each sequence $\mathcal{S}_k$ does not contain a 3-term arithmetic progression. This follows from the fact that each sequence is a constant plus $\mathcal{S}_0$ and $\mathcal{S}_0$ does not contain a 3-term arithmetic progression.

2. Sequences $\mathcal{S}_k$ do not overlap. Consider an element $a_i$ in $\mathcal{S}_0$. We start by showing that $a_i + n$, for any integer $n > 1$, does not belong to $\mathcal{S}_0$. Indeed, when we add an integer $n > 1$, we either get 2 as the last digit, or we have a carry. A carry always generates a two. That means, in any sequence $\mathcal{S}_k$ no two numbers differ by an integer more than 1. That means $\mathcal{S}_k$ and $\mathcal{S}_j$ do not overlap for any $k \neq j$.

3. Every integer-like number belongs to one of the sequences. This can be proven by showing that every integer-like number in base 3/2 comes down to a number with only 1 and 0 by subtracting 2. Indeed, if a number contains a 2, then we can always subtract a 2 from it and get a positive integer-like number. We continue subtracting 2 while we have a 2 in the number. As this process is finite we have to end with an integer-like number consisting only of ones and zeros in their base 3/2 representation.

4. We know that $\mathcal{S}_0$ is lexicographically first. We proceed by induction. Suppose for $j \leq k$ sequence $\mathcal{S}_j$ is lexicographically first by value in the set 
$\mathcal{S} \setminus \cup_{i \geq j} \mathcal{S}_i$. Consider lexicographically first sequence $F$ in this set $\mathcal{S} \setminus \cup_{i \geq k} \mathcal{S}_i$. Notice that every element in $F$ has to contain a 2 in its base 3/2 representation. That means if, we subtract a 2 from every element of $F$ we get integer-like numbers in the set $\mathcal{S} \setminus \cup_{i \geq k-1} \mathcal{S}_i$. It has to be lexicographically first, so it has to equal $\mathcal{S}_k$. Thus, sequence $F$ has to equal $\mathcal{S}_{k+1}$.
\end{proof}

We later need one more property about sequences $\mathcal{S}_i$ written in base 3/2. But first a definition. We say that a set of numbers with digits 0, 1, and 2 satisfies \textit{two-out-of-three property} if the following holds:
\begin{itemize}
\item The last digit of every number in the set uses two out of three possible digits.
\item Numbers in the set that end with $x$ can have only two possibilities for a digit before $x$, and both possibilities are realized.
\end{itemize}

\begin{lemma}
Sequences $\mathcal{S}_n$ when written in base 3/2 satisfies two-out-of three property.
\end{lemma}

\begin{proof}
Sequence $\mathcal{S}_0$ consists of numbers using zeros and ones. Thus it satisfies the two-out-of three property. Sequence $\mathcal{S}_n$ is constructed by adding the same number $x$ to all elements of $\mathcal{S}_0$ considered in base 3/2. 

Now we start from the last digit and use induction. The last digit has two possibilities: the last digit of $x$ and the last digit of $x+1$. Consider numbers in $\mathcal{S}_n$ that end with the same string of $m$ digits denoted here by $z$. When we subtract $x$ from all these numbers we get a set of numbers with a fixed string $y$ at the end. Before it, we can only have 0 or 1 as a digit. Now when we add $x$ to these numbers, we have exactly two possibilities for the number before the string. Both of them are realized.
\end{proof}

\subsection{Write $\mathcal{S}_i$ in base 3/2 and interpret them in base 3}

First recall a famous fact about 3-term integer arithmetic progressions.

\begin{lemma}
The last digits of a 3-term arithmetic progression written in base 3 are either all the same or all different.
\end{lemma}

Before proceeding we need the following statement about sequences $\mathcal{S}_i$.

\begin{lemma}
Sequence $\mathcal{S}_k$ written in base 3/2 and then interpreted in base 3 is 3-free.
\end{lemma}

\begin{proof}
Suppose sequence $\mathcal{S}_k$ written in base 3/2 and then interpreted in base 3 contains an arithmetic progression $a,b,c$. There are only two possibilities for the last digit. That means $a,b,c$ have the same last digit in base 3. We subtract this digit and divide by 3. We get numbers $a',b'c'$ that are numbers $a,b,c$ without the last digit. They have to form the arithmetic progression. By our two-out-of-three property, as the last digit is fixed, there are only two possibilities for the digit before it, which is now the last digit in the new progression $a',b'c'$. It follows that the last digit in $a',b'c'$ is the same for all three numbers. By continuing, we get that the numbers $a,b,c$ are equal to each other, leading to a contradiction.
\end{proof}

The following statement which we did not prove is the last step that is needed for our conjecture: Each sequence $\mathcal{S}_n$ is lexicographically first 3-free sequence on the available numbers in base 3 order.

Now we state our main conjecture.

\begin{conjecture}
Sequence $\mathcal{S}_k$ written in base 3/2 but interpreted in base 3 is sequence $S_k$.
\end{conjecture}

\begin{corollary}
The Stanley cross-sequence can be defined as following: Take even numbers, write them in base 3/2, interpret the resulting string as numbers written in ternary. 
\end{corollary}

Now we want to spend some time discussing sequences $\mathcal{S}_i$, for $i = 1,2,3$ in more detail.

\subsection{Examples}

We can easily describe the first few sequences $\mathcal{S}_i$ in terms of their representation in base 3/2:

\begin{itemize}
\item $\mathcal{S}_0$ are numbers written with 0 and 1: 0, 1, 10, 11, 100, and so on.
\item $\mathcal{S}_1$ are numbers that contain exactly one 2 that might be followed by zeros: 2, 12, 20, 102, 112, 120, 200, and so on.
\item $\mathcal{S}_2$ are numbers such that the last digit is 1 or 2 and the rest is a substring from $\mathcal{S}_1$: 21, 22, 121, 122, 201, 202, and so on.
\item $\mathcal{S}_3$ are numbers such that the last two digits are from the set $\{10,11,20,21\}$ and the rest is a substring from $\mathcal{S}_1$. Equivalently, we can say that $\mathcal{S}_3$ has 0 and 1 as the last digit and the rest as $\mathcal{S}_2$: 210, 211, 220, 221, 1210, 1211, and so on.
\end{itemize}

\section{Different ways to write numbers in base 1.5}\label{sec:differentways}

Non-integer bases have been known for a long time. Consider a number $\beta > 1$. The value of $x=d_n\dots d_2d_1d_0.d_{-1}d_{-2}\dots d_{-m}$ is
\[\begin{aligned}
x&=\beta ^{n}d_{n}+\cdots +\beta ^{2}d_{2}+\beta d_{1}+d_{0}\\
&\qquad +\beta ^{-1}d_{-1}+\beta ^{-2}d_{-2}+\cdots +\beta ^{-m}d_{-m}.
\end{aligned}
\]
This representation of $x$ is called a $\beta$-expansion and was introduced by R\'{e}nyi in 1957 \cite{R} and later studied by Parry \cite{P}. In such representations the numbers $d_i$ are non-negative integers less than $\beta$. Every real number has at least one $\beta$-expansion. 

What we study is different. In our expansions we are more flexible. In case of base 3/2 we allow digits 1 and 2 in our expansions. In case of base 1.5 we allow a non-digit H in our expansions.

There is a famous greedy algorithm to write an integer $N$ in an integer base $b$. This algorithm was adapted for $\beta$-expansions by R\'{e}nyi \cite{R}. 

We are trying to represent $N$ as $d_kd_{k-1}\ldots d_1d_0$. In this algorithm we start with finding the left-most digit $d_k$:

\begin{enumerate}
\item \textbf{Find the total number of digits.} Find the largest power $k$ so that $N \geq b^k$. Then $N$ has $k+1$ digits in base $b$. 
\item \textbf{Find the left-most digit.} Then the digit $d_k$ is equal $\lfloor N/b^k \rfloor$. 
\item \textbf{Repeat.} Consider the new integer $N - d_kb^k$ and continue recursively using zeros when necessary. 
\end{enumerate}

This algorithm chooses the lexicographically largest expansion in case there are several expansions. There are other expansions that choose digits 0 and 1. They were studied by many authors, starting with Kempner \cite{K}.

\textbf{Use only zeros and ones.} We find $k$ so that $1.5^k < N < 1.5^{k+1}$. The first non-zero digit is $d_k=1$. Repeat with the difference. For example, 
\[2= 10.010000010010010100000000010000001\ldots.\]
This is the same representation as used by R\'{e}nyi \cite{R} that is the largest expansion lexicographically. Another famous expansion is the lazy expansion:
\[2=0.11111\ldots.\]

As we can also use H for a digit, the number of possible expansions increases. Here are some natural examples. 

\begin{enumerate}
\item \textbf{A finite expansion for integers.} The base is designed in such a way that any integer has a finite expansion. For example, 
\[2 = 1H.\]
We show later in this section how to represent any integer as a finite expansion.
\item \textbf{Use only Hs and zeros.} Find $k$ so that $\frac{1}{2}\cdot 1.5^k < N < \frac{1}{2}\cdot 1.5^{k+1}$. The first non-zero digit is $d_k$. If follows that $N < 1.5^k$. Therefore, $d_k = $H. In this representation every number is represented using only 0 and H. For example,
\[2=\text{H}000.0\text{H}00\text{H}00\text{H}000\text{H}000\text{H}00\text{H}00\text{H}\ldots.\]
This extension chooses the first digit as far to the left as possible. It gives the largest expansion lexicographically when using digits 0, 1, and H.
\item \textbf{Use only Hs and ones.}  We can use the fact that 1=0.HHH... which we represent as 0.$\overline{\mbox{H}}$. Similarly $1.5 = \text{H}.\overline{\mbox{H}}$, and so on. A $k$-th power of 1.5, for $k > 0$, is represented with $k$ digits H before radix followed by $\overline{\mbox{H}}$. If $k<0$, then $1.5^k$ is represented with $k$ zeros after the period followed by Hs. Consider a positive real number $x$ such that $1.5^k\leq x <1.5^{k+1}$. Denote the difference $x-1.5^k$ by $d$. We can represent $d$ using only zeros and Hs as above. Moreover, the largest non-zero digit of $d$ in this representation is less than $k$ as $d <  1.5^{k+1}-1.5^k=0.5\cdot 1.5^k$. When we add these representations of $d$ and $1.5^k$, we get $x$ represented with only Hs and ones starting from the significant digit. For example, if $x = 2$, then $k=1$ and $d=0.5$. Now we need to represent 0.5 using only Hs, that is $d= $H. Summing up, with this algorithm 2 is represented as 
\[2=1.\text{HHHHHHHHHHHHHHH}\ldots\]
\item \textbf{Make the first digit as close to the number as possible.} Consider the set of numbers, $\frac{1}{2}\cdot 1.5^k$, $1.5^k$. Pick the largest number from this set that does not exceed $N$. Repeat with the difference. With this algorithm 2 is represented as 
\[2= \text{H}000.00100000\text{H}000\text{H}000\text{H}\ldots.\]
\end{enumerate}

Notice, that the second algorithm and R\'{e}nyi's algorithm are connected. If we represent $x$ using only digits 0 and H, then by replacing H with 1 we get R\'{e}nyi's representation of $2x$ with only 0 and 1.

We now look at the numbers that are not extended past the decimal points. We already studied such numbers in base 3/2. Similar to that, we call such numbers in base 1.5 \textit{integer-like}. An integer-like number is either an integer or a non-integer, but every integer has an integer-like representation.

It is well-known that while infinite $\beta$-expansions are not unique, finite $\beta$-expansions are unique. A similar statement is true in base 1.5.

\begin{lemma}
No two integer-like numbers in base 1.5 have the same integer-like representation.
\end{lemma}

\begin{proof}
Consider two different representations $a_ka_{k-1}\ldots a_0$ and $b_jb_{j-1}\ldots b_0$ of the same value. If $a_0 = b_0 \neq 0$, we can replace $a_0$ and $b_0$ in both numbers by zero, and still have two different representations with the same value. If $a_0 = b_0 = 0$, we can remove the last digit from both numbers while still having two different numbers with the same value. That means, we can assume that $a_0 \neq b_0$. Given that these numbers have the same value, we write:
\[\sum_{i=0}^k a_k\frac{3^i}{2^i} = \sum_{i=0}^k b_k\frac{3^i}{2^i},\]
where we replace H with 1/2. 
Now we multiply both sides by $2^{k+1}$. We get
\[\sum_{i=0}^k 2a_k3^i2^{k-i} = \sum_{i=0}^k 2b_k3^i2^{k-i}.\]
Both sides are integers. By taking this equation modulo 3, we get that
\[2^{k+1}a_0 = 2^{k+1}b_0.\]
Therefore,
\[a_0 = b_0.\]
This creates a contradiction with our assumption and proves the lemma.
\end{proof}

The sequence of natural numbers in base 1.5 starts as 0, 1, 1H, 1H0, 1H1, 1H0H, 1H10, 1H11, 1H0HH. 

Algorithm for incrementing a number by 1 in Base 1.5:

\texttt{Start at the last digit:}

\texttt{if 0: change to 1, stop}

\texttt{if 1: change to H, carry 1}

\texttt{if H: change to 0, carry 1}

\texttt{move one digit to the left}

In other words, start with the first zero to the right. If there are no zeros, pad one zero at the beginning. Then for the string of digits starting and including the last zero replace them in a cyclic order: $0 \rightarrow 1 \rightarrow H \rightarrow 0$. For example, the representation of 7 is 1H11. We pad it with 0, to get 01H11, then cycle these digits to 1H0HH, which is the base 1.5 representation of 8. As another example, the representation of 22 is 1H100H1, so 23 must be  1H1010H.

\section{Isomorphism}\label{sec:isomorphism}

There are a lot of similarities between bases 3/2 and 1.5. Is there a connection between the two bases? 

\begin{theorem}
Every number in base 1.5 is the same as the number with 2 times its value in base $\frac{3}{2}$ except with the digits 0, H, 1 replaced by 0, 1, 2 correspondingly.
\end{theorem}

\begin{proof}
We start by noticing that multiplying 0, H, 1 by 2 produces 0, 1, 2 correspondingly. Now consider number $x$ in base 1.5. That means we can use the base 1.5 representation to represent $x$ as a sum of powers of $3/2$ with coefficient 0, 1/2, and 1. Multiplying by 2 we get a representation of $2x$ as a sum of powers of $3/2$ with coefficient 0, 1, and 2. As sums of powers match the representations in bases and the representation is unique, the theorem is proven.
\end{proof}

For example, 2 in base 1.5 is 1H. That means 4 in base 3/2 is 21. In addition, we know other ways to represent 2 in base 1.5. We can use them to show how 4 can be represented in base 3/2. 
\begin{enumerate}
\item \textbf{Only 0 and 1.} 
\[4=\text{1}000.0\text{1}00\text{1}00\text{1}000\text{1}000\text{1}00\text{1}00\text{1}\ldots.\]
\item \textbf{Only 0 and 2.} 
\[4= 20.020000020020020200000000020000002\ldots,\]
\item \textbf{Only 1 and 2.} 
\[4=2.1111111111111111111111\ldots,\]
\item \textbf{The smallest leftover.} 
\[4= 1000.00200000100010001\ldots.\]
\end{enumerate}

Because of the isomorphism, we can rewrite the properties in Section~\ref{sec:base32} in base 3/2 in our base 1/5:

\begin{itemize}
\item Every integer only uses digits 0, H, and 1.
\item Every positive integer starts with 1.
\item Every integer bigger than 3 starts with 1H followed by either 0 or 1.
\item The last digit repeats in a cycle of 3 numbers, the last two digits repeat in a cycle of 9 numbers, and so on. The last $k$ digits repeat in a cycle of $3^k$ numbers.
\item Removing one or several last digits of an integer in this base gives another integer in the base.
\end{itemize}

Here are some other parallels. Sequence A244040 is defined as a sum of digits of $n$ in fractional base 3/2. With respect to fractional base 1.5, A244040$(2n)$ is twice the sum of digits of $n$.

Also, sequence A256785 is related to base 1.5. It shows the numbers that have an integer digit sum in base 1.5: 1, 5, 11, 14, 20, 21, 22, 23, 26, 29, 30, 31, $\ldots$. Equivalently, these are numbers that have an even number of digits H in their representation. On the other hand, $2\cdot$ A256785 is the sequence of even numbers that have an even sum in base 3/2.

As another example, consider sequence A320035 that provides indices of integers in integer-like numbers written in ascending order in base 1.5. The same sequence is sequence of indices of even integers in the sequence of integer like numbers written in ascending order in base 3/2. The sequence for all integers in base 3/2 is now sequence A320272:
\[0, 1, 3, 6, 11, 17, 25, 34, 46, 60, 77, 96, 117, \ldots.\]

%0, 1, 3, 6, 11, 17, 25, 34, 46, 60, 77, 96, 117, 142, 169, 200, 232, 268, 308, 353, 401, 453, 508, 568, 631, 699, 771, 847, 929, 1015, 1108, 1205, 1308, 1415, 1527, 1644, 1767, 1896, 2029, 2169, 2315, 2468, 2626, 2790, 2961, 3140, 3325, 3519, 3719, 3924, 4138, 4357, 4585, 4817, 5057, 5305, 5561, 5823, 6094, 6372, 6658, 6951, 7251, 7562, 7880, 8208, 8543, 8890, 9245, 9609, 9982, 10366, 10760, 11162, 11572, 11991, 12419, 12858, 13305, 13762, 14226, 14698, 15181, 15673, 16177, 16689, 17209, 17742, 18285, 18840, 19404, 19976, 20560, 21153, 21760, 22378, 23007, 23646, 24297, 24960, 25637, 26325, 27027

By the way, if we use the dictionary order, then indices or integers in integer-like numbers in base 3/2 are sequence A261691: 
\[0, 1, 2, 6, 7, 8, 21, 22, 23, 63, 64, 65, 69, 70, 71, 192, 193, 194, 207, 208, \ldots.\]

The similar sequence in base 1.5 is every other term in A261691:
\[0, 2, 7, 21, 23, 64, 69, 71, 193, 207, \ldots.\]
It is our mystery sequence A265316.

\section{Acknowledgements}
We are grateful to PRIMES STEP program for allowing us to do this research.

\begin{comment}
\bigskip
\hrule
\bigskip

\noindent 2010 {\it Mathematics Subject Classification}: Primary ; Secondary .

\noindent \emph{Keywords:} .

\bigskip
\hrule
\bigskip

\noindent (Concerned with sequences \seqnum{}, and )
\end{comment}

\end{document}